\documentclass[12pt]{article}

\usepackage[applemac]{inputenc}
\usepackage[T1]{fontenc}
\usepackage{amsmath}
\usepackage{amsfonts}
\usepackage[english]{babel}

\usepackage{geometry}                
\usepackage[parfill]{parskip}    
\usepackage{graphicx}
\usepackage{amssymb}
\usepackage{epstopdf}
\usepackage[all]{xy}

\usepackage{amsthm} 
\newtheorem{thm}{Theorem}[section] 
\theoremstyle{definition} 
\newtheorem{dfn}{Definition}[section] 
\theoremstyle{axiom}

\theoremstyle{remark} 
 
\theoremstyle{plain} 

\newtheorem{prop}[thm]{Proposition}
\newtheorem{cor}[thm]{Corollary}

\theoremstyle{plain} 
 
\begin{document}

\title{On complemented, uniquely complemented and uniquely complemented nondistributive lattices \\
 (a historical and epistemological note about a mathematical mystery)}
\author{Daniel Parrochia}
\date{University of Lyon (France)}
\maketitle

\textbf{Abstract.}

Complemented lattices and uniquely complemented lattices are very important, not only in mathematics, but also in physics, biology, and even in social sciences. They have been investigated for a long time, especially by Huntington, Birkhoff, Dilworth and others. And yet, on some of these structures - namely, uniquely complemented nondistributive lattices -, despite the many existing articles concerning them, we basically know very little. In this article, we situate these lattest structures in the context of complemented and uniquely complemented lattices, offering a general overview of the links between these lattices and others, close to them, such as the orthocomplemented lattices of physics as well as various other partially ordered sets. We finally show how uniquely complemented nondistributive lattices have been constructed with the technique of free lattices.\\


\section{Introduction}

Let us start with a non mathematical problem. Suppose we want to define an object or an entity $O$ whose set of properties $P$ are not accessible to us. Suppose however that there exists correlatively a set of properties $E$ such that $P \subset E$. Depending on the structure of $E$, it may seem possible to access $P$ from the subset of properties of $E$ which are known to us and which define, by hypothesis, an object $M$. For example, if $E$ is a Boolean lattice, where every element has a unique complement, then it is clear that for any property $p' \in E -P$ there exists a property $p \in P$ such that $(p')' = p$. We can thus hope to be able to define the object $O$, from the properties of the object $M$\footnote{Philosophically, this was the approach of negative theologians (see \cite{Par1}, 196-199). But this could also be the approach of physicists when, for example, it comes to describing a particle that they do not yet know, from the properties (mass, energy level, spin, etc.) of a few others.}. 

Due to a result of Dilworth (see \cite{Dil2}) obtained in 1945, which asserts that any lattice - including nondistributive lattices - can be embedded in a uniquely complemented lattice, a question may nevertheless arise: the condition of being boolean can't be weakened? If so, what is the general form of the $E$-lattice of properties and can it be precisely described?

E. V. Huntington, in his 1904 paper (see \cite{Hun}, 305)\footnote{For a historical commentary on Huntington's axiomatics (see \cite{Bar}).}, posed a similar problem which, in modern terms, may be expressed as follows: is every uniquely complemented lattice distributive or must one add, for that, additional conditions?

\section{Huntington's conjecture}

Huntington himself have conjectured\footnote{The conjecture is not explicit, but, at the end of the passage quoted from page 305 of his article, where the problem is stated, Huntington referred to postulate 19 of the second series of his postulates describing the Boolean algebra, which effectively characterizes this one: if $a \leq b$ then conversely $\bar{b} \leq \bar{a}$.} that the property of being "uniquely complemented" was sufficient to cause distributivity, and thus, define a Boolean lattice. As Gr\"{a}tzer has shown (see \cite{Gra1}; \cite{Gra2}), the conjecture was still widely accepted in the 1930s, reinforced - as it seemed to be - by a number of results of the following form: 
\begin{equation}
if\ a\ lattice\ L\ is\ uniquely\ complemented\ and\ has\ property\ X, then\ L\ is\ boolean.
\end{equation}
$X$ could mean different things: for example, to be modular, or to be relatively complemented (see \cite{Ber})\footnote{Let us roughly say, for the moment - a more precise definition will come later - that "relatively complemented" means: for all $a \leq b \leq c$, there is a unique $d$ with $b \wedge d = a$ and $b \vee d = c$.} or even to be orthocomplemented\footnote{"Orthocomplementation" is the name of complementation when applied to subspaces of a vector space (see def. 3.11).} , results known to Birkhoff and Von Neumann since the end of the 1930s (see \cite{Bir2}) and exposed in \cite{Bir1}. 

But $X$ could still mean "to be of finite dimension" (see \cite{Dil1}) or even "to be complete, atomic and dually atomic" (see \cite {Bir3})\footnote{A lattice $L$ is complete if $\bigwedge X$ and $\bigvee X'$ exist for any subset $X$ of $L$; a lattice $L$ is atomic if $ L$ has an element 0 and if, for all $a \in L, 0 < a$, there is an element $p$ such that $p \leq a$ and $p$ is an atom, that is, if $0 <p$. A dual atom is a meet-irreducible element whose upper cover is 1. A lattice is called dually atomic  (or dually atomistic) if each of its elements is a meet of dual atoms (see \cite{Ste}).}.

Huntington had himself initiated this series of restrictions by showing (in modern terms) that if $(B, \wedge, \vee, ')$ is a uniquely complemented lattice, then, by adding to it the property $x \wedge y = 0 \Rightarrow y \leq x'$, i.e. by making the complementation a bit special, $(B, \wedge, \vee, ')$ becomes a Boolean algebra, solution at the origin of many other proposals of the same kind.

In 1945 however, R. P. Dilworth (see \cite{Dil2}) showed that the "Huntington conjecture" was wrong. Through a very 
long and technical article, he proved, on the one hand, that there exists a uniquely complemented nondistributive lattice and, on the other hand, a much more general result:
\begin{equation}
Every\ lattice\ can\ be\ embedded\ into\ a\ uniquely\ complemented\ lattice.
\end{equation}
So it is a fact: a uniquely complemented and yet nondistributive lattice can exist. However, as Gr\"{a}tzer wrote (see \cite{Gra2}, 701), three problems remain:

\begin{enumerate}
 \item As all known examples of nondistributive uniquely complemented lattices are freely generated one way or another, a first question could be: is there a construction of a nondistributive uniquely complemented lattice that is different?
\item In the same vein, we could also ask : is there a "natural" example of a nondistributive uniquely complemented lattice from geometry, topology, or whatever else?
\item And finally, we would like - if possible - to answer this crucial question: is there a complete example of a nondistributive uniquely complemented lattice?
\end{enumerate}

We would like to advance a little on these three questions.

\section{Some definitions}

Let us first recall some definitions:

\begin{dfn}[Poset]
A {\it partially ordered set (or poset)} ($P, \leq$) consists of a nonempty set $P$ and a binary relation $\leq$ on $P$ such that $\leq$ satisfies the following properties:
\begin{enumerate}
\item $x \leq x$ \hspace{11\baselineskip} (Reflexivity);
\item $ x \leq y$ and $y \leq x$ implies that $x \leq x$ \quad (Antisymmetry);
\item $ x \leq y$ and $y \leq z$ implies that  $x \leq z$ \quad (Transitivity).

If $\leq$ satisfies also:

\item $x \leq y$ or $y\leq x$ \hspace{8\baselineskip} (Linearity)

then $P$ is a {\it chain}.
\end{enumerate}
\end{dfn}

\begin{dfn}[Bounded poset]
A bounded poset is one that has  both a bottom element 0 and a top element 1.
\end{dfn} 

\begin{dfn}[Bounds]
Let $H \subseteq P$ and $x \in P$. A bound $x$ of $H$ is the least upper bound of $H$ iff, for any upper bound $y$ of $H$, we have $x \leq y$. Similarly, a bound $w$ of $H$ is the greatest lower bound of $H$ iff, for any lower bound $z$ of $H$, we have $z \ge w$.
\end{dfn}

\begin{dfn}[Lattice]
A poset $(L, \leq)$ is a lattice if the great lower bound inf($x,y$) and the upper lower bound sup($x,y)$ exist for all $x, y \in L$.
\end{dfn}

\begin{dfn}[Modular lattice]
A lattice $L$ is said to be {\it modular} if:
\[
\forall x, y, z \in L: (x \leq z) \Rightarrow (x \vee (y \wedge z) = (x \vee y) \wedge z)
\]
\end{dfn}

\begin{dfn}[Orthomodular lattice]
A lattice with a zero 0 and a one 1 in which for any element $x$
 there is an orthocomplement $x^\perp$, i.e. an element such that;

1) $x \vee x^\perp =1, x \wedge x^\perp = 0, (x^\perp)^\perp = x$,

2) $x \leq y \Rightarrow x^\perp \ge y^\perp$,

and such that the orthomodular law:

3) $x \leq y \Rightarrow y = x \vee(y \wedge x^\perp)$

is satisfied, is said to be {\it orthomodular}.
\end{dfn}

\begin{dfn}[Distributive lattice]
A lattice $L$ is said to be {\it distributive} if:
\[
\forall x, y, z \in L, x \wedge (y \vee z) = (x \wedge y) \vee (x \wedge z)
\]
\textnormal{and/or}:
\[
x \vee (y \wedge z) = (x \vee y) \wedge (x \vee z).
\]
\end{dfn}

\begin{dfn}[Bounded lattice]
A bounded poset which is a lattice is a bounded lattice.
\end{dfn}

\begin{dfn}[Complement of an element]
If $L$ is a bounded lattice, then we say that $y \in L$ is a complement of $x \in L$ if $x \wedge y = 0$ and $x \vee y = 1$. In this case, we say that $x$ is a complemented element of $L$. Clearly, every complement of a complemented element is itsef complemented.
\end{dfn}

\begin{dfn}[Complemented lattice]
A lattice $L$ is a complemented lattice if every element of $L$ is complemented.
\end{dfn}

\begin{dfn}[Orthocomplemented lattice]
An orthocomplemented lattice is a complemented lattice in which every element has a distinguished complement, called an orthocomplement. This orthocomplement behaves like the complementary subspace of a subspace in a vector space.
\end{dfn}

Formally, let $L$ be a complemented lattice and denote $M$ the set of complements of elements of 
$L.\ M$ is clearly a subposet of $L$, with $\leq$ inherited from $L$. For each $a \in L$ let $M_{a} \subseteq M$
 be the set of complements of $a.\ L$ is said to be orthocomplemented if there is a function $\perp: L \to M$, called an orthocomplementation, whose image is written $a^\perp$ for any $a \in L$, such that:

1. $a^\perp \in M_{a}$,

2. $(a^\perp)^\perp =a$ and

3. $\perp$ is order-reversing; that is, for any $a, b \in L, a \leq b$ implies $b^\perp \leq a^\perp$.

The element $a^\perp$  is called an orthocomplement of $a$ (via $\perp$).

\begin{dfn}[Relatively complemented lattice]
A lattice $L$ is a relatively complemented lattice if every interval [$x,  y$] of $L$ is complemented. A complement in [$x, y$] of $a \in [x, y]$ is called a relative complement of $a$.
\end{dfn}

\begin{dfn}[Uniquely complemented lattice]
A complemented lattice $L$ in which complements are unique (that is, for all $x \in L$ there exist $x' \in L$ such that $(x')' = x$) is called a uniquely complemented lattice. 
\end{dfn}

\begin{dfn}[Residuated lattice]
A residuated lattice\footnote{This structure has been introduced by Morgan Ward and Robert Palmer Dilworth  is 1939 (see \cite{War}).  It is helpful to think of the last two operations as left and right division and thus the equivalences can be seen as "dividing" on the right by $b$ and "dividing" on the left by $a$. The study of such objects originated in the context of the theory of ring ideals in the 1930s. The collection of all two-sided ideals of a ring forms a lattice upon which one can impose a natural monoid structure making this object into a residuated lattice.} is an ordered algebraic structure $\langle L, \wedge, \vee, \centerdot, e, \setminus, /\rangle$
such that $\langle L, \wedge, \vee \rangle$ is a lattice, $\langle L,\centerdot, e\rangle$ is a monoid, and $\setminus$ and / are binary operations for which the equivalences
\[
a . b \leq c \Leftrightarrow a \leq c/b \Leftrightarrow b \leq a \setminus c
\]
hold for all $a, b, c \in L$.
\end{dfn}

\begin{dfn}[Atom]
If $L$ is a lattice with bottom element 0, then by an atom of $L$ we mean an element $a$ such that $0 \leq a$. If for every $x \in L \setminus \{0\}$ there is an atom $a$ such that $a \leq x$, then we say that $L$ is {\it atomic}.
\end{dfn}

\section{Some well-known theorems}

We present here, without their proofs, some well-known theorems which will be useful later.

\begin{thm}[\cite{Bir1}]
Any distributive lattice is modular.
\end{thm}

The two following theorems come from (\cite{Gra2}, 62):

\begin{thm}
In a bounded distributive lattice, an element can have only one complement (so a bounded distributive lattice is a boolean lattice).
\end{thm}

\begin{thm}
In a bounded distributive lattice, if $a$ has a complement, then it also has a relative complement in any interval containing it.
\end{thm}

\begin{thm}[\cite{Bly}, 78]
In a distributive lattice, all complements and relative complements that exist are unique.
\end{thm}

\begin{thm}[\cite{Bir3}]
Every uniquely complemented atomic lattice is distributive.
\end{thm} 

\begin{thm}[\cite{Bir4}]
A lattice $L$ is distributive if and only if there exists no sublattice $A \subseteq L$ isomorphic to either $M_{3}$ or $N_{5}$.
\end{thm}

\section{Non distributivity and complementation}

A lattice in which each element has at most one complement may have elements with no complement at all. It is rather easy to come up with nondistributive lattices with that property, even if we require that there is at least one pair of elements which are complements of one another.

A trivial example is a bounded nondistributive lattice in which no element except 0 and 1 has a complement.

For less trivial cases - some in which there are other pairs of complements - we can take the following examples:

\subsection{Example 1}

We begin wit a complemented nonmodular lattice (see Fig. 1).
\begin{figure}[h] 
    \vspace{-0.5\baselineskip}
	   \hspace{9\baselineskip}	  
	   \includegraphics[width=8in]{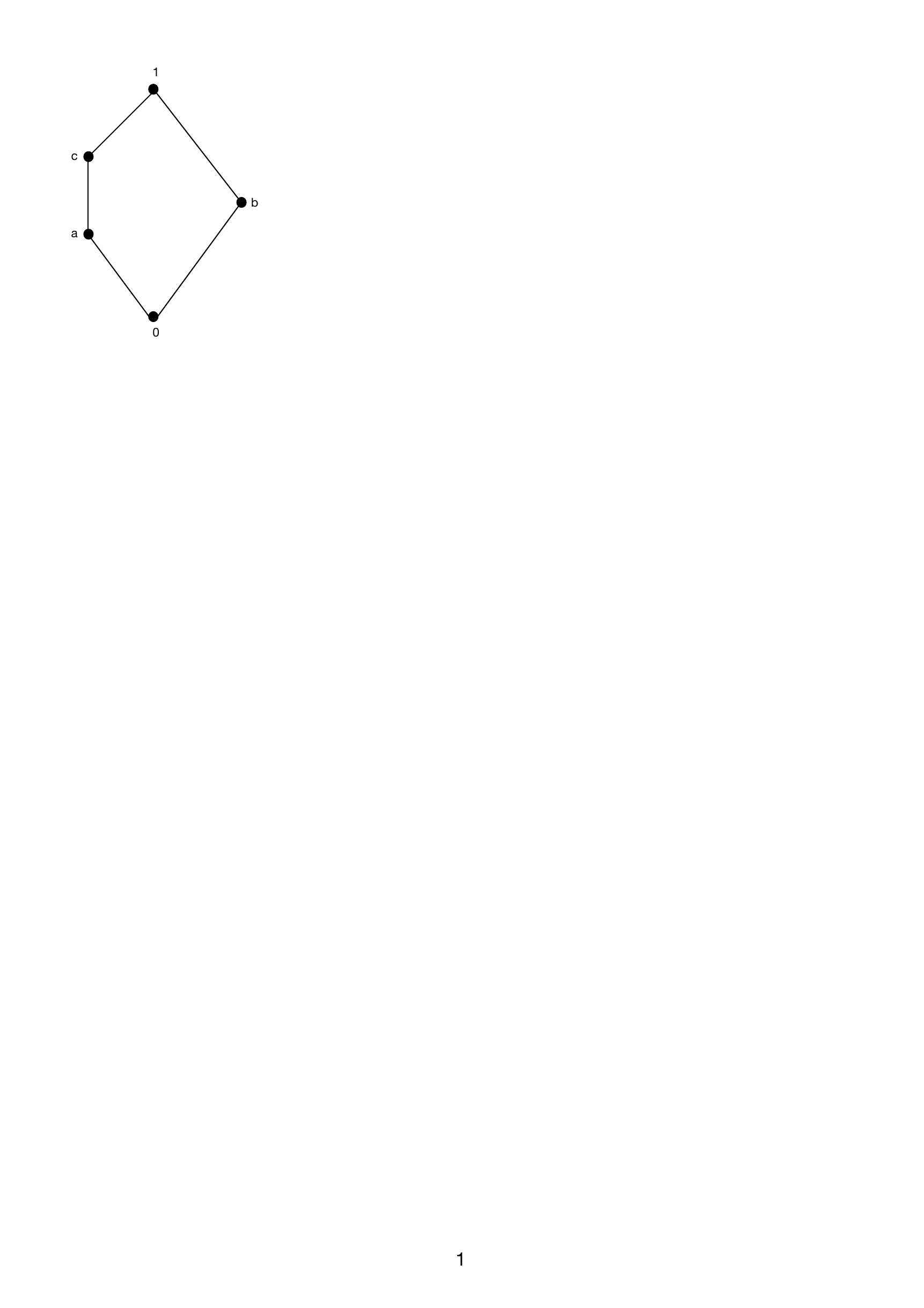} 
	      \vspace{-42\baselineskip}
	   \caption{The complemented nonmodular lattice $N_{5}$}
	   \label{fig: treillis11}
	\end{figure}
	
In this lattice, we can verify that $a \wedge b = 0$ and $a \vee b = 1$, so $\bar{a}^1 = b, \bar{b}^1 = a$.

We also have: $b \wedge c = 0$ and $b \vee c = 1$, so $\bar{b}^2 = c, \bar{c}^1 = b$.

Besides, we obviously have: $0 \wedge 1 = 0$ and $0 \vee 1 = 1$, so $\bar{0} = 1$ and $\bar{1} = 0$.

This lattice is therefore complemented. But we can show that it is not modular:
\[
b \leq d : b \vee(c \wedge d) = b \vee a = b,
\]
\[
\ \ \qquad (b \vee c) \wedge d = e \wedge d = d.
\]
So it is not distributive.

\subsection{Example 2}

Consider now the lattice of Fig. 2:

\begin{figure}[h] 
    \vspace{-0.5\baselineskip}
	   \hspace{9\baselineskip}	  
	   \includegraphics[width=8in]{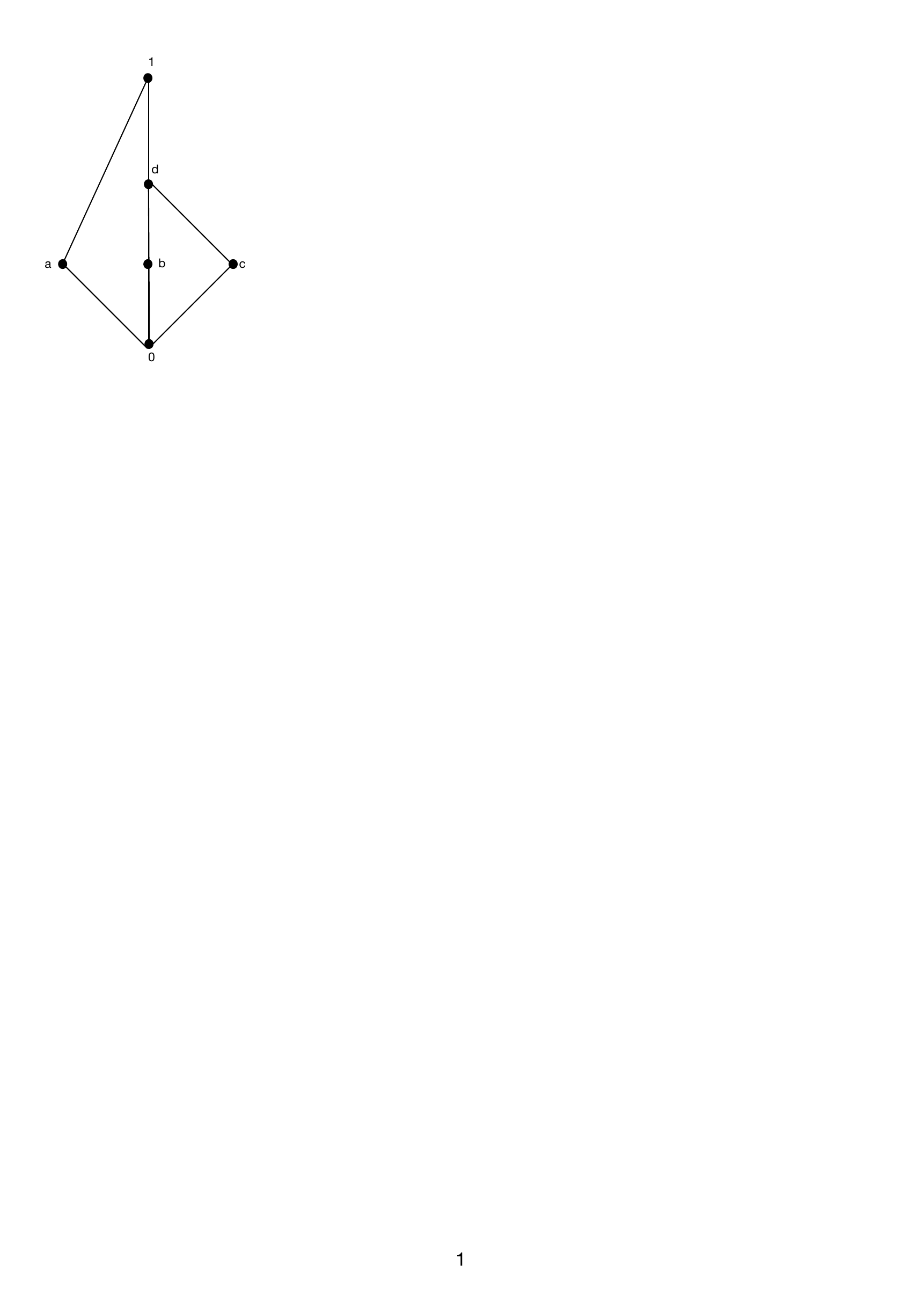} 
	      \vspace{-42\baselineskip}
	   \caption{A complemented but nondistributive lattice}
	   \label{fig: treillis22}
	\end{figure}
	
As before, 0 is the complement of 1 and vice versa. Let's move on to the other elements. We have, in particular:

 $a \wedge b = 0$ and $a \vee b = 1$, so $\bar{a}^1 = b, \bar{b}^1 = a$,
 
  $a \wedge c = 0$ and $a \vee c = 1$, so $\bar{a}^2 = c, \bar{c}^1 = a$,
  
   $a \wedge d = 0$ and $a \vee d = 1$, so $\bar{a}^3 = d, \bar{d}^1 = a$.
   
  The lattice is complemented ($a$ has three complements). But it is not distributive. It is easy to see that:
  \[
  a \vee (b \wedge c) = a \vee 0 = a \neq (a \vee b) \wedge a \vee c) = 1 \wedge 1 = 1,
  \]
  and:
  \[
  c \wedge (a \vee b) = c \wedge 1 = c \neq (c \wedge a) \vee (c \wedge b) = 0 \vee 0 = 0.
  \]
  
  \subsection{Example 3}
  
  Let us now give an example of a complemented modular but nondistributive lattice. This is the famous lattice $M_{3}$ (see Fig. 3).
                
                \begin{figure}[h] 
    \vspace{-1\baselineskip}
	   \hspace{7\baselineskip}	  
	   \includegraphics[width=7in]{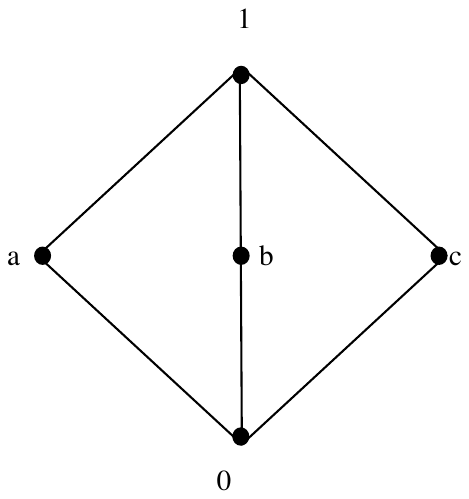} 
	      \vspace{-13\baselineskip}
	   \caption{The nondistributive lattice $M_{3}$}
	   \label{fig: treillis33}
	\end{figure}
	           
         We have, in particular:
         
         $a \wedge b = 0, a \vee b = 1$, so $\bar{a}^1 = b$ and $\bar{b}^1 = a$,         
         
           $a \wedge c = 0, a \vee c = 1$, so $\bar{a}^2 = c$ and $\bar{c}^1 = a$,
           
              $b \wedge c = 0, b \vee c = 1$, so $\bar{b}^2 = c$ and $\bar{c}^1 = a$.       
       
       As before, we also have obviously: $\bar{0} = 1$ and $\bar{1} = 0$.

\subsection{Example 4}                         
                               
In Fig. 4,  (0, 1) and $(a,b)$ are the only pairs of complements, and the lattice is not distributive, since it has $M_{3}$ as a sublattice.

  \begin{figure}[h] 
    \vspace{-1\baselineskip}
	   \hspace{7\baselineskip}	  
	   \includegraphics[width=7in]{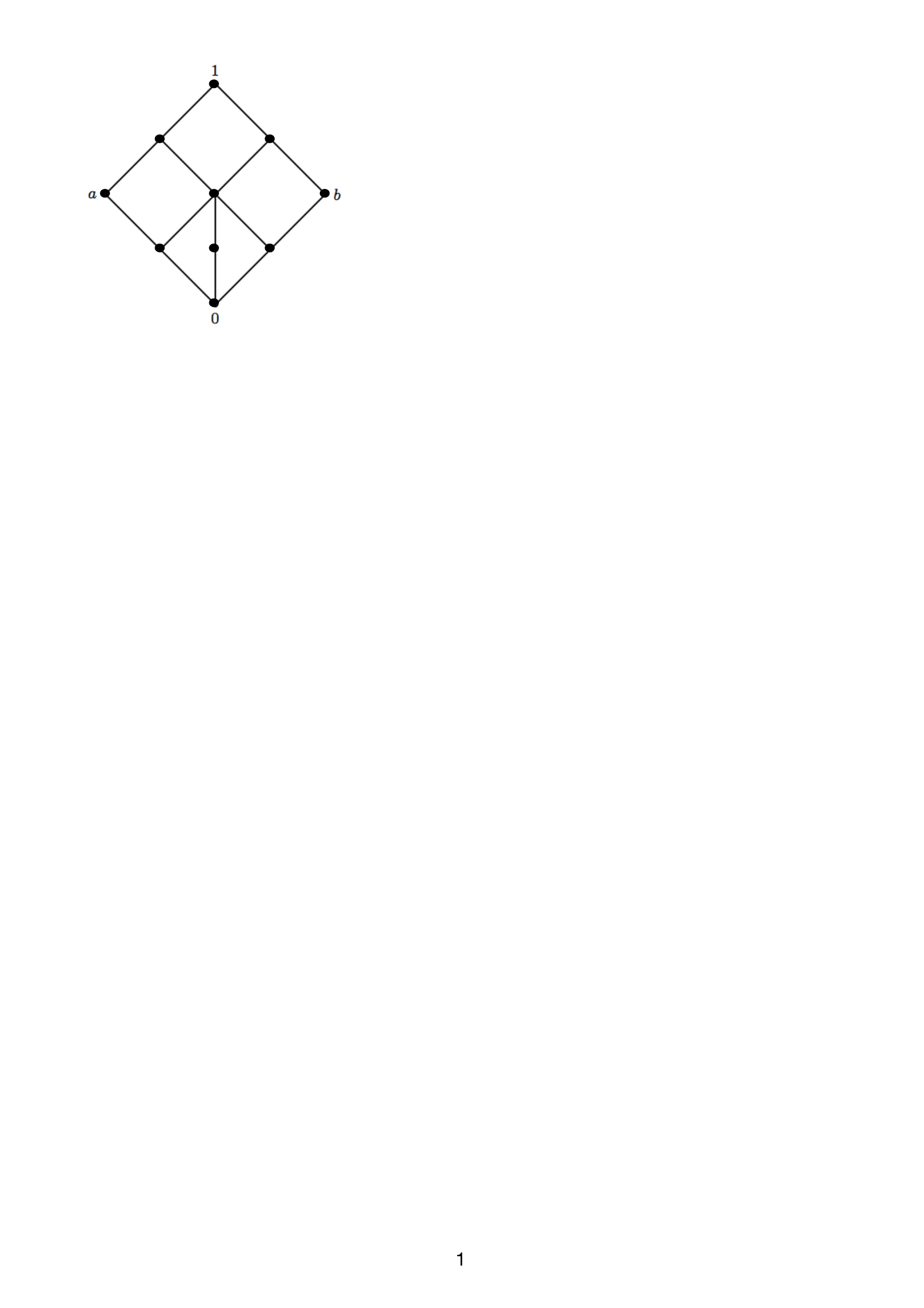} 
	      \vspace{-38\baselineskip}
	   \caption{A nondistributive lattice with $M_{3}$ as sublattice}
	   \label{fig: treillis44}
	\end{figure}
	
\subsection{Example 5}

The orthocomplemented lattices or, as we sometimes say, the {\it ortholattices} (used in quantum mechanics) are generally nondistributive. The non-distributivity results from the inclusion relations between vector subspaces of the three-dimensional vector space, which lead to a nondistributive lattice. Now the logical relations between the propositions of quantum mechanics - for example, those which describe the spin of a particle - fit into a lattice analogous to the lattice of the vector subspaces of three-dimensional space. Fig. 5 (B) gives an example of an orthocomplemented nondistributive lattice.
	
	\begin{figure}[h] 
    \vspace{-0.5\baselineskip}
	   \hspace{-1\baselineskip}	  
	   \includegraphics[width=7in]{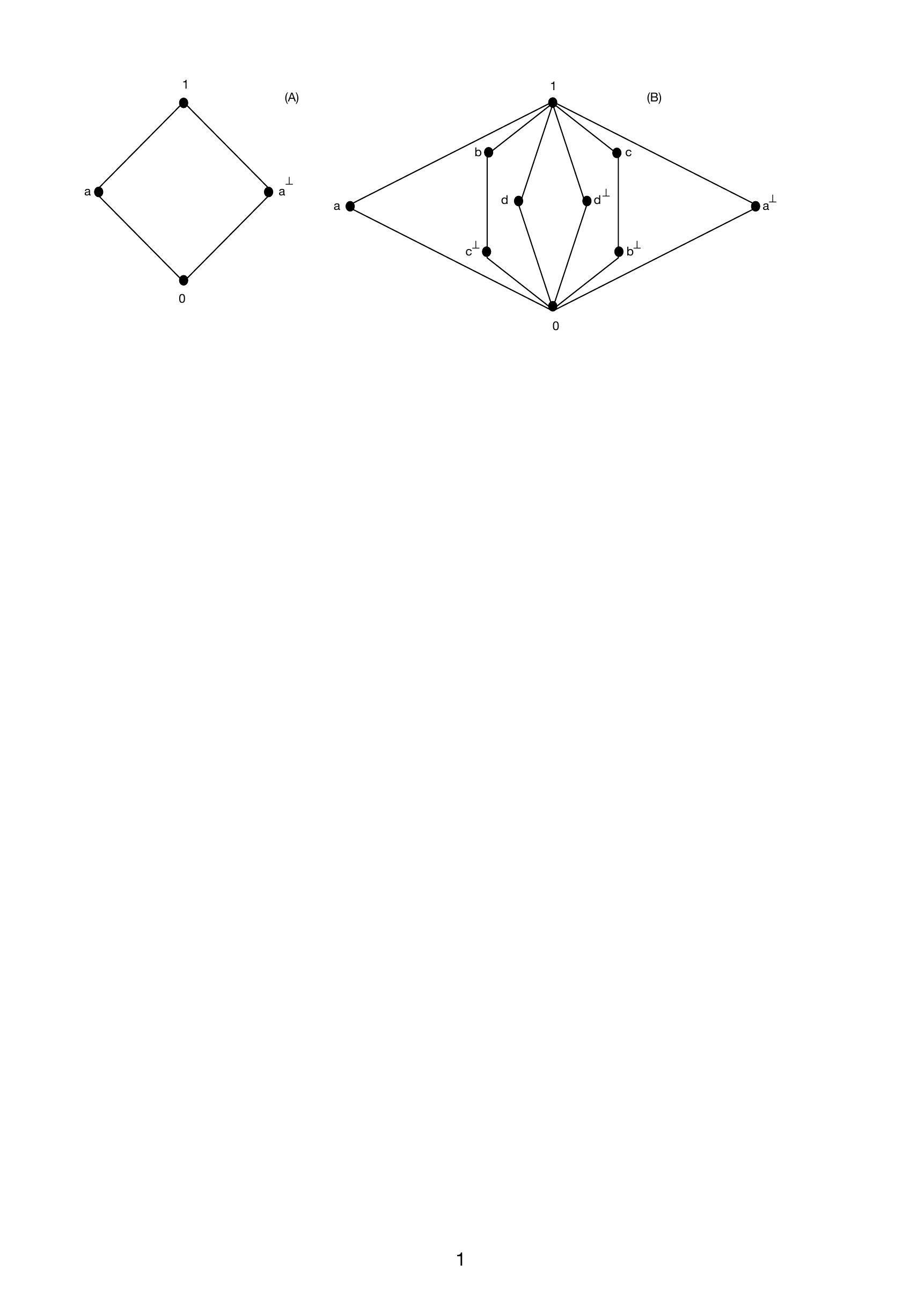} 
	      \vspace{-38\baselineskip}
	   \caption{Orthocomplemented distributive (A) and nondistributive (B) lattices}
	   \label{fig: treillis55}
	\end{figure}
	
	Orthocomplemented lattices are always complemented, but not necessarily uniquely complemented. Orthocomplemented lattices that are uniquely complemented are also distributive, and so Boolean (see Fig. 5 (A)). This kind of lattices is associated with classical mechanics, not with quantum mechanics.
		
\section{Graph of interactions between some of the main lattices}

If we recapitulate what we know about the previous lattices, we obtain the following propositions (most of them can be found in \cite{Rut}, if not in \cite{Bir1}).

1. A Boolean algebra (BA) is a complemented distributive lattice (DL) (this is the definition). It is also a Heyting algebra (HA)  and an orthocomplemented lattice (OCL).

2. A distributive orthocomplemented lattice (DOL) is orthomodular (OML) (so a Boolean algebra is orthomodular).

3. An orthomodular lattice (OML)  is orthocomplemented and an orthocomplemented lattice (OCL)  is complemented (CL). So an orthomodular lattice is complemented (but not necessarily uniquely complemented).

4. A complemented lattice (CL) is bounded (BL). 

5. A Heyting algebra (HA) is bounded (BL) and residuated (RL).

6. A distributive lattice (DL) is modular (ML).

7. A modular complemented lattice (MCL) is relatively complemented (RCL) (so a Boolean algebra is relatively complemented).

8. A Heyting algebra (HA) is a distributive lattice (DL).

9. A totally ordered set (TOS) is a distributive lattice (DL).

10. A modular lattice (ML) is semimodular SML) (so a distributive lattice is semimodular).

11. A semimodular lattice (SML)  is atomic (AL) (so a modular lattice is atomic).

12. An atomic lattice (AL) is a lattice (L). 

13. A lattice (L) is a semilattice (SL). 

14. A semilattice (SL) is a partially ordered set (POS). 

All this lead to the graph of Fig. 6.

\begin{figure}[h] 
 \vspace{-4\baselineskip}
  	   \hspace{0\baselineskip}	  	  
	   \includegraphics[width=7in]{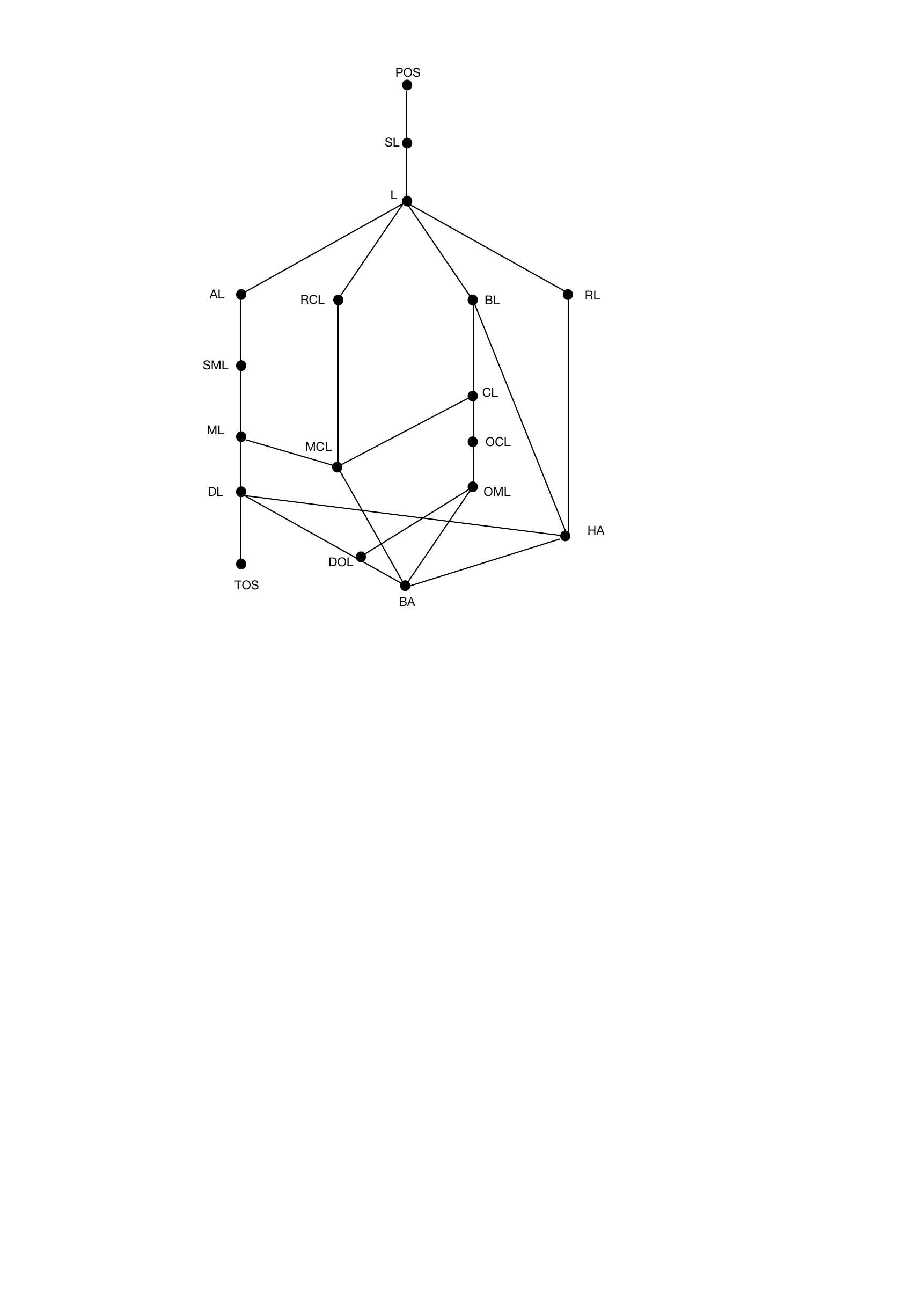} 
	      \vspace{-27\baselineskip}
	   \caption{A graph of some of the main lattices}
	   \label{fig: treillis55}
	\end{figure}

\section{Uniquely complemented nondistributive lattices}

Now if we are looking for a nondistributive lattice in which every element has a unique complement, that is not easy to point out. However, we can get some informations from above:

\subsection{First properties}

\begin{prop}
 A nondistributive lattice uniquely complemented is necessarily non atomic.
\end{prop}

\begin{proof}
According to theorem 4.5, every uniquely complemented atomic lattice is distributive. So we can deduce that a nondistributive lattice is non atomic or non uniquely complemented (or both). But suppose now $L$, a nondistributive uniquely complemented lattice. If it were atomic, it would exists an atomic uniquely complemented lattice which would be nondistributive, a contradiction. So $L$ is necessarily non atomic. 
\end{proof}

We also know two other theorems concerning the subject:

\begin{thm}[\cite{Bly}]
Every uniquely complemented lattice of finite width is distributive.
\end{thm}

Hence we can infer the following:

\begin{prop}
 A uniquely complemented nondistributive lattice is necessarily of infinite width.
\end{prop}

We also have:

\begin{thm}[\cite{Neu}]
Every uniquely complemented modular lattice is distributive.
\end{thm}

Hence, we get:

\begin{prop}
A uniquely complemented nondistributive lattice is not modular.
\end{prop}

We can also recall the following theorem from Sali\u{\i} (see \cite{Sal1}):

\begin{thm}[\cite{Sal1}]
A compactly generated lattice with unique complement is distributive.
\end{thm}

\begin{prop}
A uniquely complemented nondistributive lattice cannot be compactly generated.
\end{prop}

By propositions 7.1, 7.3, 7.5 and 7.7 we can deduce:

\begin{prop}
A uniquely complemented nondistributive lattice is necessarily non atomic, non modular, non compactly generated, and of infinite width.
\end{prop}

As Sali\u{\i} observes, if to the requirement for unique complements, we add the condition that the lattice is modular or that the Morgan laws hold, or that the lattice is atomic, in all these cases, we find that the lattice operations are distributive (see \cite{Sal2}). Therefore, it is not surprising that,  in other circumstances, we get negative results: a uniquely complemented nondistributive lattice, of course, is none of that.

\subsection{Sali\u{\i}'s properties of distributive and nondistributive complemented Lattices}

Reading pages 39-57 of \cite{Sal2}, allows us to further detail the properties of distributive lattices and - by contrast - nondistributive complemented lattices in the following way:

Recall first what is the ascending (resp. descending) chain condition\footnote{The ascending chain condition (ACC) and descending chain condition (DCC) are finiteness properties 
satisfied by some algebraic structures, especially ideals in certain commutative rings. These conditions played an important role in the development of the structure theory of commutative rings in the works of David Hilbert, Emmy Noether, and Emil Artin. But the conditions themselves can be stated in an abstract form, so that they make sense for any partially ordered set. A poset which satisfies both the ascending and descending chain condition is said to be well-founded and converse well-founded.} in a partially ordered set.

\begin{dfn}[ACC]
A partially ordered set (poset) $P$ is said to satisfy the ascending chain condition (ACC) if no infinite strictly ascending sequence:
\[
a_{1} < a_{2} < a_{3} < ...
\]
of elements of $P$ exists. Equivalently, every weakly ascending sequence:
\[
a_{1} \leq a_{2} \leq a_{3} \leq ...
\]
of elements of $P$ eventually stabilizes, meaning that there exists a positive integer $n$ such that:
\[
a_{n} = a_{n+1} = a_{n+2} = ...\ .
\]
\end{dfn}

\begin{dfn}[DCC]
Similarly, $P$ is said to satisfy the descending chain condition (DCC) if there is no infinite descending chain of elements of $P$.
 Equivalently, every weakly descending sequence:
\[
a_{1} \ge a_{2} \ge a_{3} \ge ...
\]
of elements of $P$ eventually stabilizes.
\end{dfn}

Sali\u{\i} proves the following theorems:

\begin{thm}
If a uniquely complemented lattice (UCL) satisfies the descending chain condition or the ascending chain condition, then it is distributive.
\end{thm} 

\begin{thm}
If the largest element of a UCL is the least upper bound of its sets of atoms, then the lattice is distributive.
\end{thm} 

\begin{thm}
If in a uniquely complemented lattice there exists for each nonzero element a prime ideal that does not contain it, then this lattice is distributive.
\end{thm}

\begin{thm}
A uniquely complemented lattice $L$ is distributive if and only if for each nonzero element $x$ there exists a homomorphism $\phi: L \to \{0,1\}$ onto the two-element lattice such that $\phi(x) = 1$.
\end{thm}

\begin{dfn}
A lattice with 0 is called initially complemented (or sectionally complemented).
\end{dfn}

\begin{thm}
An initially complemented UC lattice is distributive.
\end{thm}

\begin{thm}
A uniquely complemented lattice is distributive if and only if:
\begin{equation}
(x \vee ((x \vee y) \wedge x')) \wedge y < 0,
\end{equation}
for any x and any $y >0$.
\end{thm} 

\begin{thm}
In a uniquely complemented lattice, the quasi-identity $x \leq y \Rightarrow x' \ge y'$  implies distributivity.
\end{thm}

From the above theorems, we can infer:

1. If a lattice is non distributive, which is the case of the uniquely complemented nondistributive lattice UCN), then it does not satisfies the descending chain condition or the ascending chain condition.

2. In a UCN the largest element is not the least upper bound of its sets of atoms.

3. In a UCN, for each nonzero element, all prime ideals contain it.

4. In a UCN, for each nonzero element $x$, there does not exist a homomorphism $\phi: L \to \{0,1\}$ onto the two-element lattice such that $\phi(x) = 1$.

5. A UCN is not an initially complemented UC lattice.

6. In a UCN, the inequation (3) does not hold for any $x$ and any $y >0$.

7. In a UCN, the quasi-identity $x \leq y \Rightarrow x' \ge y'$ does not imply distributivity.

Sali\u{\i} proves also:

\begin{thm}
A uniquely complemented ortholattice is distributive.
\end{thm}

\begin{thm}
A uniquely complemented orthomodular lattice is distributive.
\end{thm}

\begin{thm}
In a uniquely complemented lattice either of the following conditions implies distributivity:
\begin{equation}
(x \wedge y)' = x' \vee y'\ \textnormal{for any comparable element}\ x\ \textnormal{and}\ y,
\end{equation}
\begin{equation}
(x \wedge y)' = x' \vee y'\ \textnormal{\ for any noncomparable element}\ x\ \textnormal{and}\ y.
\end{equation}
\end{thm}

Hence we can infer:

8. A UCN can neither be a complemented ortholattice, nor a complemented orthomodular lattice.

9. In a UCN, the conditions (4) and (5) do not hold.

Moreover, the Birkhoff-von Neumann theorem asserts that the uniquely complemented modular lattices are distributive. But it turns out that they are distributive in a much wider class.
According to Sali\u{\i}:

\begin{thm}
A uniquely complemented 0-semimodular lattice is distributive.
\end{thm}

\begin{thm}[\cite{Gri}]
A uniquely complemented 0-modular lattice is distributive.
\end{thm} 

Hence we can infer that a UCN is neither 0-semimodular nor 0-modular.

Among several other properties implied by a uniquely complemented lattice, we extract this one, reported by Gr\"{a}tzer in his review of Sali\u{\i}'s book (see \cite{Gra4}).

\begin{thm}
A uniquely complemented lattice is distributive iff each nonzero element contains a nonzero regular element\footnote{Recall that an element $a$ is "regular" iff
$a \wedge x = 0$ and $a \wedge y = 0$ imply that $a \wedge (x \vee y) = 0$.}. 
\end{thm}

So, here again, we can state the following proposition:

\begin{prop}
 In a uniquely nondistributive complemented lattice each element does not contain a nonzero regular element.
\end{prop}

Sali\u{\i} recognizes that we know very little about uniquely complemented lattices. Among the positive properties of uniquely complemented lattices with comparable complements (UCC lattice in short), we still have the following:

 \begin{thm}[\cite{Sal2}]
 Let $L$ be a UCC lattice. If $a < b \in L$, then $b \wedge a>0$.
 \end{thm}
 
In the next theorem, "$a <b$" denotes that "a is covered by b", that is, $a < b$ and there is no element in between.

 \begin{thm}[\cite{Sal2}]
 Let $L$ be a UCC lattice. If $a <b$, then $c = b \wedge a'$ is an atom, and it is a relative complement of $a$ in $[0, b]$.
 \end{thm}

Such a theorem shows, for instance, that if a UCC lattice is atomic, then it is dually atomic, and that the covering relation is preserved at least under some joins and meets.

Concerning the Dilworth's lattice - the L15 object of Kwuida's classification (see \cite{Kwu}, 157) - its positive properties would be semi-complementation, dual-semi-complementation, complementation and weak-complementation (see also \cite{Kal}). Kwuida further adds that this lattice is infinite, that its complementation is automatically an involution, and that it cannot be antitone nor relatively complemented.

However, these observations are about complementation, not about non-distributivity. Moreover, the existence of the uniquely complemented nondistributive lattice remains a mystery, and this, although Sali\u{\i}, using a result of Adam and Sichler (see \cite{Ada}), has proved in his book that there is a variety\footnote{The concept of {\it lattice variety} (see Definition 7.4) is very special. Historically, and according to \cite{Jip}, "the study of lattice varieties evolved out of the study of varieties in general, which was initiated by Garrett Birkhoff in the 1930's. He derived the first significant results in this subject, and further developments by Alfred Tarski and later, for congruence distributive
varieties, by Bjarni J\'{o}nsson, laid the groundwork for many of the results about lattice
varieties. During the same period, investigations in projective geometry and modular
lattices, by Richard Dedekind, John von Neumann, Garrett Birkhoff, George Gr\"{a}tzer,
Bjarni J\'{o}nsson and others, generated a wealth of information about these structures,
which was used by Kirby Baker and Rudolf Wille to obtain some structural results about
the lattice of all modular subvarieties. Nonmodular varieties were considered by Ralph
McKenzie, and a paper of him published in 1972 stimulated a lot of research in this direction.
Since then the efforts of many people have advanced the subject of lattice varieties in
several directions, and many interesting results have been obtained".} which contains it.

To understand this proposition, let us first introduce a definition of the concept of {\it lattice variety}: 

\begin{dfn}[Lattice varieties]
 Let $\mathcal{E}$ be a set of lattice identities (equations), and denote by $Mod\ \mathcal{E}$ 
the class of all lattices that satisfy every identity in $\mathcal{E}$. A class $V$ of lattices is a lattice
variety if:
\[
V = Mod\ \mathcal{E}
\]
for some set of lattice identities $\mathcal{E}$. 
\end{dfn}
The class of all lattices, which we will denote by $\mathcal{L}$, is of
course a lattice variety since $\mathcal{L} = Mod\ 0$. One also frequently encounter the following
lattice varieties:
\[
T = Mod\{x = y\}
\]
\[
V = Mod\{xy + xz = x(y + z)\}
\]
\[
M = Mod\{xy + xz = x(y + xz)\}
\]
all trivial lattices, all distributive lattices, all modular lattices.

Let us come now to the result of Sali\u{\i}:
\begin{thm}
The variety $V$ of lattices defined by the identity:
\[
(x \wedge (y \vee z)) \vee (y \wedge (x \vee t)) = ((x \wedge (y \vee z)) \vee y) \wedge (x \vee (y \wedge (x \vee t)))
\]
contains nondistributive uniquely complemented lattices.
\end{thm}
 
Anyway, as we have seen, R. P. Dilworth has already disclosed the existence of such lattices in 1945, by proving the following theorem:

\begin{thm}[Dilworth]
Every lattice is a sublattice of a lattice with unique complements.
\end{thm}

As Dilworth comments, "thus any nondistributive lattice is a sublattice of a lattice with unique complements which a fortiori is {\it not} a Boolean algebra" (see \cite{Dil2}, 123). As a nondistributive lattice cannot be a sublattice of a distributive lattice, there necessarily exists a nondistributive lattice with unique complements.

As reported by Gr\"{a}tzer, Dilworth's result has been found several times. In 1964 it was still proved by Dean (see \cite{Dea}) who "extends Dilworth's free lattice generated by an order $P$ to the free lattice generated by an order $P$ {\it with any number of designated joins and meets}" (see \cite{Gra2}, 699). But Dean's proof was the same complicated induction as the one in \cite{Dil2}. Subsequently, a greatly simplified proof can be found in H. Lakser's Ph.D. thesis, 1968 (see \cite{Lak}). Then Chen and Gr\"{a}tzer proved that we could greatly simplify the proof of Dilworth's theorem by skipping 2 of the 4 steps (the crucial steps 2 and 3), and proving a slightly stronger theorem (see \cite{Che}; \cite{Gra3})\footnote{We will develop these results in section 9.}.

As all known nondistributive uniquely complemented lattices - according to Gr\"{a}tzer - are freely generated, it is important to take a closer look at what free lattices are - what we will do now.

\section{Free lattices}
 Let us start first with a quite informal definition.
 
\begin{dfn}[\cite{Kau}]
A free lattice generated by $n$ generators $x_{1}, x_{2}, ... ,x_{n}$ is a lattice where appear {\it all} the possible combinations obtained with these $n$ distinct elements and the $\wedge$ and $\vee$ operations.
\end{dfn}

With $X_{n} = \{x_{1}, x_{2}, ... ,x_{n}\}, n \ge 3$, we can easily see, by enumeration, that the free lattices have an infinite number of elements (this has been proved by \cite{Whi2}, see below, theorem 8.1). We have, for example,
\[
(((((x_{1} \wedge x_{2}) \vee x_{3}) \wedge x_{1}) \vee x_{2}) \wedge .....),
\]
and continue like this ad infinitum. As we will see, the properties of commutativity, associativity, idempotence and absorption do not allow limiting the enumeration as soon as $n$ is equal to 3.

\subsection{Modular free lattices with $n$ generators}

If we impose the property of modularity (see definition 3.5) on the free lattice $X_{3}$, the number of its elements becomes finite and drops to 28 (a result of Dedekind, see \cite{Gra1}, 49).

\subsection{Distributive free lattices with $n$ generators}

If we now impose one or the other of the distributivity properties (see definition 3.8), the lattice only comprises 18 elements\footnote{For a proof, see \cite{Gra1}, 49.}, certain elements, among the 28 of the corresponding modular lattice, being equivalent to others as a result of the distributive property (for example, $(x_{1} \wedge x_{2}) \vee (x_{1} \wedge x_{3}) = x_{1} \wedge (x_{2}) \vee x_{3})).$

 \begin{figure}[h] 
    \vspace{-1\baselineskip}
	   \hspace{-1\baselineskip}	  
	   \includegraphics[width=7in]{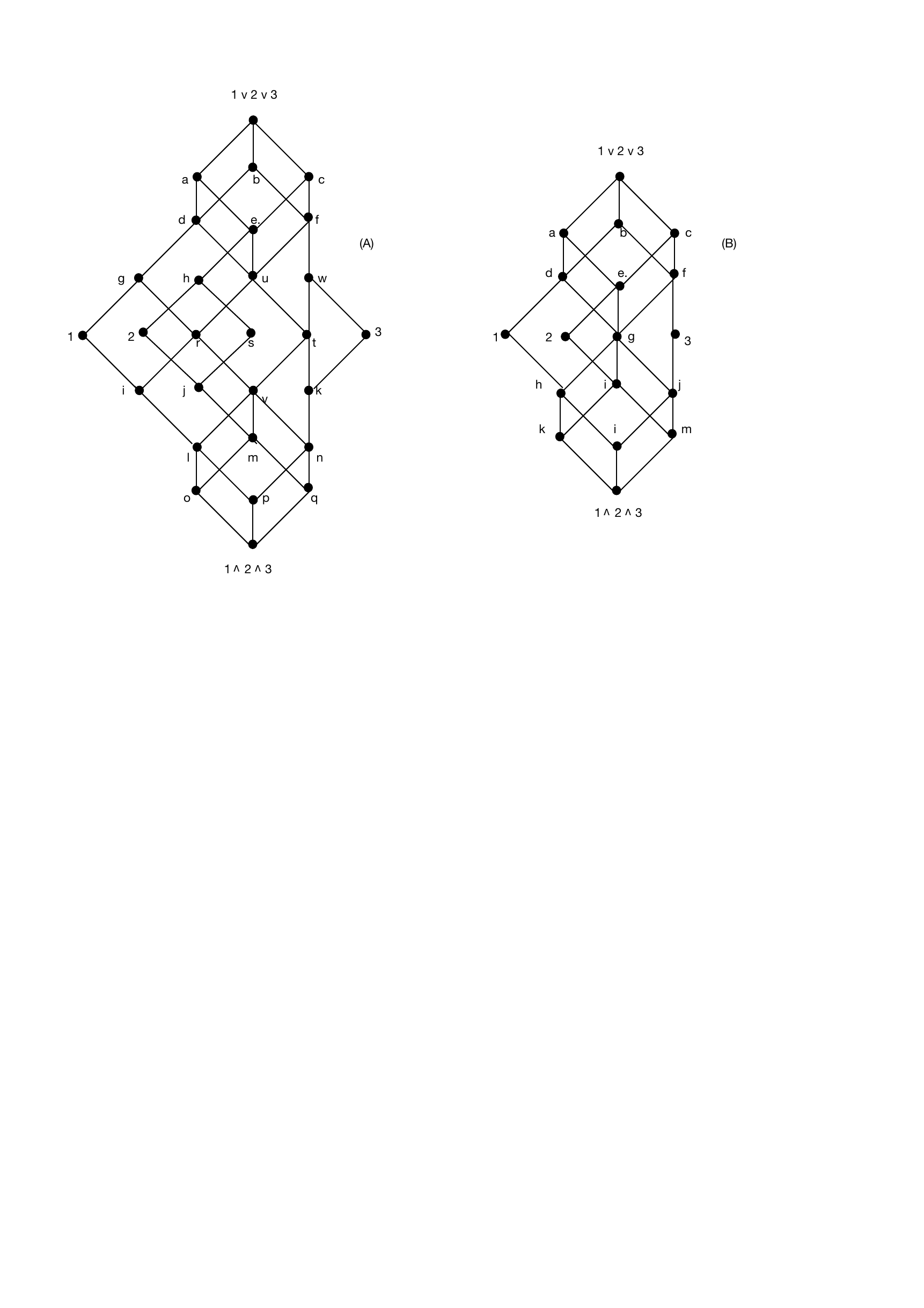} 
	      \vspace{-28\baselineskip}
	   \caption{Modular (A) and distributive (B) free lattices}
	   \label{fig: TL1}
	\end{figure}	

Legend of diagrams \\

Modular lattice (A)

$a: 1 \vee 2 ;		\hspace{12\baselineskip}					l: (1 \wedge 2) \vee (1 \wedge 3)$; 

$b: 1 \vee 3;			\hspace{12\baselineskip}				m: (1 \wedge 2) \vee (2 \wedge 3)$; 

$c: 2 \vee 3;				\hspace{12\baselineskip}			n:  (1 \wedge 3) \vee (2 \wedge 3)$; 

$d: (1 \vee 2) \wedge (1 \vee 3);\hspace{8\baselineskip}     o: 1 \wedge 2$; 

$e: (1 \vee 2) \wedge (2 \vee 3);	\hspace{8\baselineskip}	p: 1 \wedge 2$;	

$f: (1 \vee 3) \wedge (2 \vee 3);	       \hspace{8\baselineskip}	q: 2 \wedge 3$;

$g: 1 \vee (2 \wedge 3);			\hspace{10\baselineskip}	r: [1 \wedge (2 \vee 3)] \vee (2 \wedge 3)$;

$h: 2 \vee (1 \wedge 3):			\hspace{9.5\baselineskip}  s: [2 \wedge (1 \vee 3)] \vee (1 \wedge 3)$;

$i: 1 \wedge (2 \vee 3)			\hspace{10.5\baselineskip}	t: [3 \wedge (1 \vee 2)] \vee (1 \wedge 2)$.

$j: 2 \wedge (1 \vee 3);			\hspace{10\baselineskip}	u: (1 \wedge 2) \vee (2 \wedge 3) \vee (1 \wedge 3)$; 

$k: 3 \wedge (1 \vee 2);			\hspace{10\baselineskip}	v: (1 \vee 2) \wedge (2 \vee 3) \wedge (1 \vee 3)$; 

$\hspace{15\baselineskip} w: 3 \vee (1 \wedge 2)$.\\ \\. \\

Distributive lattice (B)

$a: 1 \vee 2 ;		\hspace{12\baselineskip}					h:  1 \wedge (2 \vee 3)$; 

$b: 1 \vee 3;			\hspace{12\baselineskip}				i: 2 \vee (1 \wedge 3)$; 

$c: 2 \vee 3;				\hspace{12\baselineskip}			j: 3 \wedge ( 1 \vee 2$; 

$d: 1 \vee (2 \wedge 3);      \hspace{10\baselineskip}     			k: 1 \wedge 2$; 

$e: 2 \vee (1 \wedge 3);  		  \hspace{10\baselineskip}						l: 1 \wedge 3$;

$f: 3 \vee (1 \wedge 2);		  \hspace{10\baselineskip}		 				       m: 2 \wedge 3$;

$g: (1 \wedge 2) \vee (1 \wedge 3) \vee (2 \wedge 3)
       = (1 \vee 2) \wedge (1 \vee 3) \wedge (2 \vee 3)$.

By imposing the distributive property, we check that all distributive free lattices are unique for given and finite $n$. We thus find the values of Table 1.

\begin{table}[htp]
\begin{center}
\begin{tabular}{|c|c|}
\hline
    Generators ($n$) & Elements ($N$) \\\hline
       1 &  1 \\
       2 & 42 \\
       3 & 18 \\
       4 & 166 \\
       5 & 7,579 \\
       6 & 7,828\ 532 \\
       ... & ... \\\hline
\end{tabular}
\caption{Number of generators and number of elements in the free lattice}
\end{center}
\label{default}
\end{table}%

\subsection{Enumeration of the free lattices}

\begin{dfn}
We will call {\it monomial in $\wedge$} (resp. in $\vee$) an expression where we will find only monomial in $\wedge$ (resp. in $\vee$).
\end{dfn}

\begin{dfn}
We will call {\it polynomial} an expression comprising both monomials in $\wedge$ joined by $\vee$ or, conversely, monomials in $\vee$ joined by $\wedge$.
\end{dfn}

Free lattices, especially interesting in view of generating a nondistributive uniquely complemented lattice, have been studied, as we said, by Philip M. Whitman in the 1940s. Whitman defines free lattices in \cite{Whi1}, answering the question of how to compare two "lattice polynomials" and showing that, given a polynomial, there is a shortest polynomial equal to it in the free lattice.

 Whitman also proves several important theorems in \cite{Whi2}:

\begin{thm}
$FL(3)$ has $FL(n)$ as a sublattice for any finite or countable $n$.
\end{thm}

\begin{cor}
For $n \ge 3$ and any finite or countable $k$, $FL(n)$ has $FL(k)$ as a sublattice.
\end{cor}

\begin{thm}
$FL(n)$ has no sublattice isomorphic to the free modular or distributive lattice on more than two generators.
\end{thm}

Now, recall that, when we seek to approximate the elements in a partial order by much simpler elements, we are led to a class of continuous posets, consisting of posets where every element can be obtained as the supremum of a directed set of elements that are way-below the element. If one can additionally restrict these to the compact elements of a poset for obtaining these directed sets, then the poset is even algebraic. Both concepts can be applied to lattices, so that a continuous lattice is a complete lattice that is continuous as a poset. In this context, Whitman proves the following theorem:

\begin{thm}
$FL(n)$ is continuous\footnote{The theory of continuous lattices was carried out by Dana Scott (see \cite{Sco1}, \cite{Sco2}, \cite{Sco3}) and also, for a more recent reference and the study of the Scott and Lawson topologies, \cite{Gie1}. It led to the discovery that these generalization of algebraic lattices could be used to assign meanings to programs written in high-level programming languages. On the purely mathematical side, research into the structure theory of compact semilattices led Lawson (\cite{Law}) and others (\cite{Gie2}; \cite{Hof}) to consider the category of those compact semilattices which admit a basis of subsemilattice neighborhoods at each point. It was discovered in \cite{Hof} that those objects are precisely the continuous lattices of Scott. One of the most important features of continuous lattices is that they admit sufficiently many homomorphisms (that is, mappings which preserve arbitrary infs and directed sups) into the unit interval to separate the points. The topological form of this result is due to Lawson.}.
\end{thm}

Whitman also proves:

\begin{thm}
$FL(n)$ is not complete for $n \ge 3$.
\end{thm}

\begin{thm} 
$FL(4)$ has an infinite chain of distinct elements.
\end{thm}

Whitman ended by stating a number of open problems, the most important being the following ones: do there exist infinite connected chains in $FL(n)$? Can we study the free complete lattice and the completion of $FL(n)$ by cuts? And finally, can we determine the finite sublattices of $FL(n)$?\footnote{Since that time, free lattices have been also studied by Dean, Freese, 
Je\v{z}ek, Je\v{z}ek and Slav\v{\i}k, J\v{o}nsson, J\v{o}nsson and Kiefer, J\v{o}nsson and Nation, Wille... A bibliography may be found in \cite{Fre}.}

None of the questions posed by Whitman seem to have been answered yet. In an article dating from the 1980s (see \cite{Sal4}), Sali\u{\i} noted that,"since Dilworth discovered that any lattice may be isomorphically embedded in a suitable uniquely complemented lattice (i.e. in a lattice with 0 and 1 in which every element has one and only one complement), not much about this class of close-to-Boolean lattices has become known" (see \cite{Sal3}). At this time - and still now - it was an open problem if such lattice might be complete without being distributive. Birkhoff and Ward (see \cite{Bir3}) have proved that it does not have to be atomic and in \cite{Sal1}, as we have seen, it was shown that the property of being compactly generated also implies distributivity. Another fact concerning complete uniquely complemented lattices was stated in \cite{Sal3}, namely that every such lattice is isomorphic to a direct product of a complete atomic Boolean lattice and a complete atomless uniquely complemented lattice. In this latest article, Sali\u{\i} deal with so-called regular elements in a uniquely complemented lattice. It only was shown that regular elements form a complete Boolean sublattice.

The study of free lattices is also essential because, as Adam and Sichler remarked,, "the lattices constructed by R. P. Dilworth in \cite{Dil2} contain the free lattice on countably many generators as a sublattice. Hence, in particular, any nontrivial lattice identity fails to hold in any of Dilworth's lattices" (see \cite{Ada}). By a nontrivial identity, Adam and Sichler mean an identity that does not follow from the lattice axioms. In this context, a growing conjecture has been that any uniquely complemented lattice that satisfies a nontrivial lattice identity is distributive. However, though different authors add numerous restrictions that lead to define a uniquely complemented distributive lattice, this is not indicative of the general situation: Adam and Sichler have shown that there are 2$^{\aleph_{0}}$ varieties of lattices for which Dilworth's theorem holds.

More recently, John Harding (see \cite{Har1}) was able to prove the following result:
\begin{thm}
For $\kappa$ an infinite cardinal, every complete at most UC-lattice can be regularly embedded into a $\kappa$-complete UC-lattice.
\end{thm}

The same Harding proved before (see \cite{Har2} that:

\begin{thm}
The MacNeille completion of a uniquely complemented lattice need not be complemented.
\end{thm}

Some problems on free lattices like this one - which lattices (and in particular which countable lattices) are sublattices of a free lattice? - and some other issues connected to  algorithmic may be found in \cite{Fre}.

\section{Some realizations of Dilworth's lattice}

As Schmerl says (see \cite{Sch}, 1366), less is known about infinite lattices, and it is only in the case of distributive lattices that had been significant results. In Chajda et al. (see \cite{Cha}, 31), we may read that "using free lattice techniques, Dilworth proved that every lattice is isomorphic with a sublattice of a lattice with unique complements. Hence there exists a uniquely complemented lattice containing the nonmodular $N_{5}$ as a sublattice which is, of course, not distributive". Maybe we can look at an infinite uniquely complemented lattice containing $N_{5}$ as a sublattice. Where to find a good candidate? 

\subsection{Gr\"{a}tzer's analysis of Dilworth's proof}

According to \cite{Gra2}, all known examples of nondistributive uniquely complemented lattices are freely generated one way or another. 

It is obviously the case of Dilworth's lattice itself. The four steps of its construction are the following: 

A. $P$ being a partially ordered set, the first step consists in imbedding $P$ in a lattice $L$ so that bounds, whenever they exist, of pairs of elements of $P$ are preserved. $L$ is namely the free lattice generated by $P$ in the sense that the only containing relations in $L$ are those which follow from lattice postulates and preservation of bounds (the methods employed are an extension of those used by Whitman in the study of free lattices) (see Fig. 8).

 \begin{figure}[h] 
    \vspace{-2\baselineskip}
	   \hspace{5\baselineskip}	  
	   \includegraphics[width=7in]{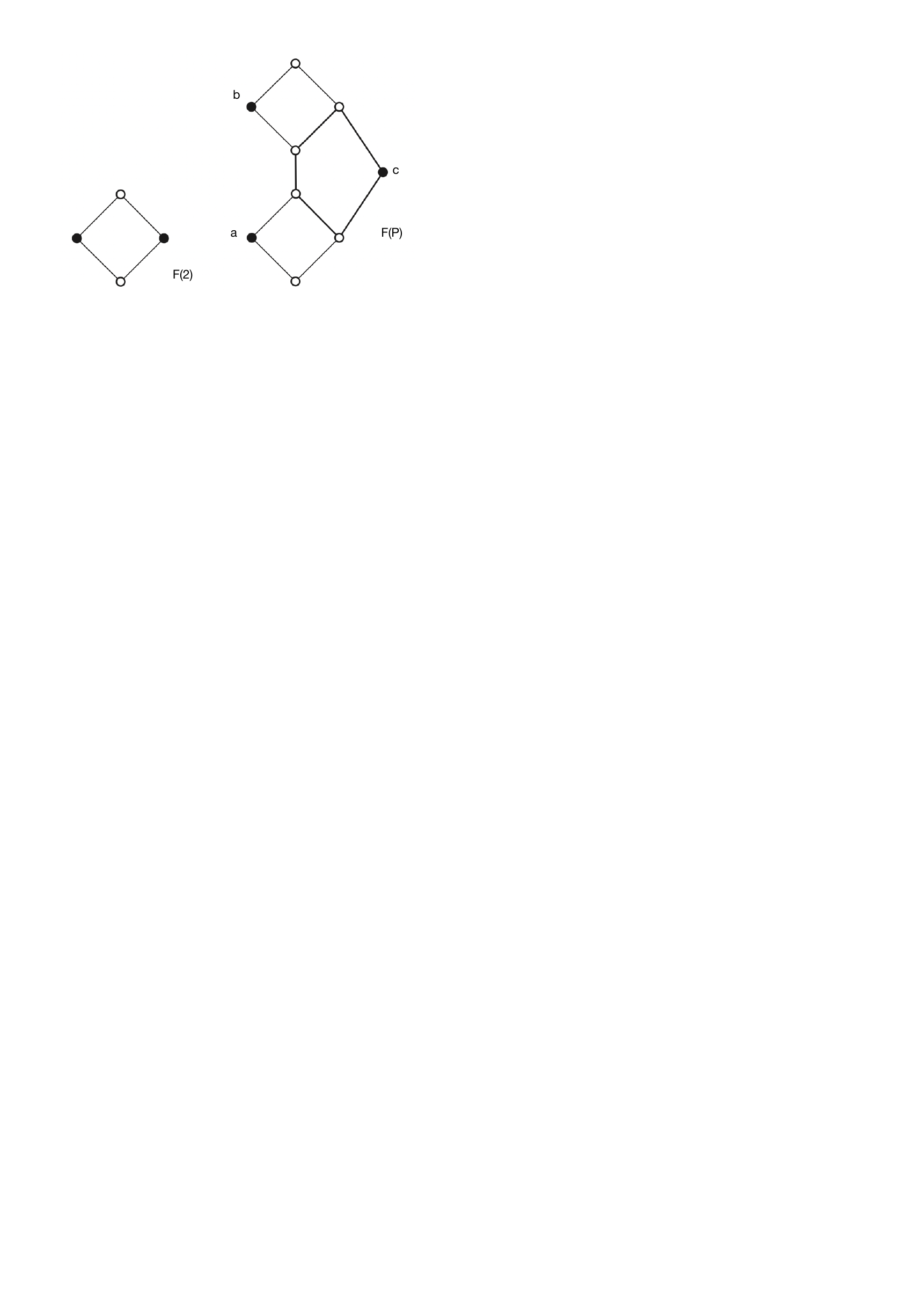} 
	      \vspace{-39\baselineskip}
	   \caption{Example of free lattices generated by an order}
	   \label{fig: Dilworth1}
	\end{figure}	

B. Next, the lattice $L$ is extended to a lattice $O$ with unary operator, that is, a lattice over which an operation $a^*$ is defined with the property:
\begin{equation}
a =b\ \textnormal{implies}\ a^* = b^*.
\end{equation}
The only containing relations in $O$ are those which follow from lattice postulates, preservation of bounds, and from (6). 

C. In the third step, a sublattice $N$ is selected from $O$ over which a new operation $a^*$ is defined for which (6) holds and also having the property:
\begin{equation}
 (a^*)^* = a.
 \end{equation}
Thus $N$ is a lattice with reflexive, unary operator. $N$ is again free in the sense that the only containing relations in $N$ are those which follow from lattice postulates, preservation of bounds, and the two properties (6) and (7).

D. Finally, a homomorphic image $M$ of $N$ is constructed in which the operation $a^*$ becomes a complementation. It follows from the structure theorems of $O$ that this complementation is unique. Furthermore, $M$ contains $P$ and is indeed the free lattice with unique complements generated by $P$.

In the end, one gets the following theorem:
\begin{thm}
The free lattice with unique complements generated by two elements contains as a sublattice the free lattice with unique complements generated by a denumerable set of elements.
\end{thm}
This proves the difference between free lattices with unique complements and Boolean algebras, since the free Boolean algebra generated by a finite number of elements is always finite.

\subsection{Dean's theorem and Lakser's simplification}

Dean's Theorem (see \cite{Dea}) has extended Dilworth's free lattice generated by an order $P$ to the free lattice generated by $P$ {\it with any number of designated joins and meets}. Then a greatly simplified proof can be found in Lakser's Ph.D. thesis, 1968 (see \cite{Lak}). But Chen and Gr\"{a}tzer, eliminating section 2 and 3 of Dilworth's article, found another proof in two steps in 1969. In fact, the result they proved is much stronger that Dilworth's one. This supposes the following definitions:

\begin{dfn}[Almost uniquely complemented lattice]
A lattice $L$ is said {\it almost uniquely complemented} if it is bounded and every element has at most one complement. 
 \end{dfn}
 
 \begin{dfn}[$\{0,1\}$-embedding]
 A $\{0,1\}$- embedding is an embedding that maps the 0 to 0 and the 1 to 1.
\end{dfn}

They proved the following theorem:

\begin{thm}
Let $L$ be an almost uniquely complemented lattice. Then $L$ can be $\{0, 1\}$-embedded into a uniquely complemented lattice.
\end{thm}

\subsection{Gr\"{a}tzer and Lakser's solution}

 In 2006, Gr\"{a}tzer and Lakser provided the one-step following solution:

Let $K$ be a bounded lattice. Let $a \in K-\{0,1\}$, and let $u$ be an element not in $K$. The authors extend the partial ordering $\leq$ of $K$ to $Q=K \cup \{u\}$ by $0 \leq u \leq 1$. They also extend the lattice operations $\wedge$ and $\vee$ to $Q$ {\it as commutative partial meet and join operations}. 

For $x \leq y$ in $Q$, they define $x \wedge y = x$ and $x \vee y=y$. In addition, they let $a \wedge u = 0$ and $a \vee u=1$, as can be seen in Fig. 9. 

\begin{figure}[h] 
\vspace{-1\baselineskip}  
 \hspace{9\baselineskip}	
    	   \includegraphics[width=7in]{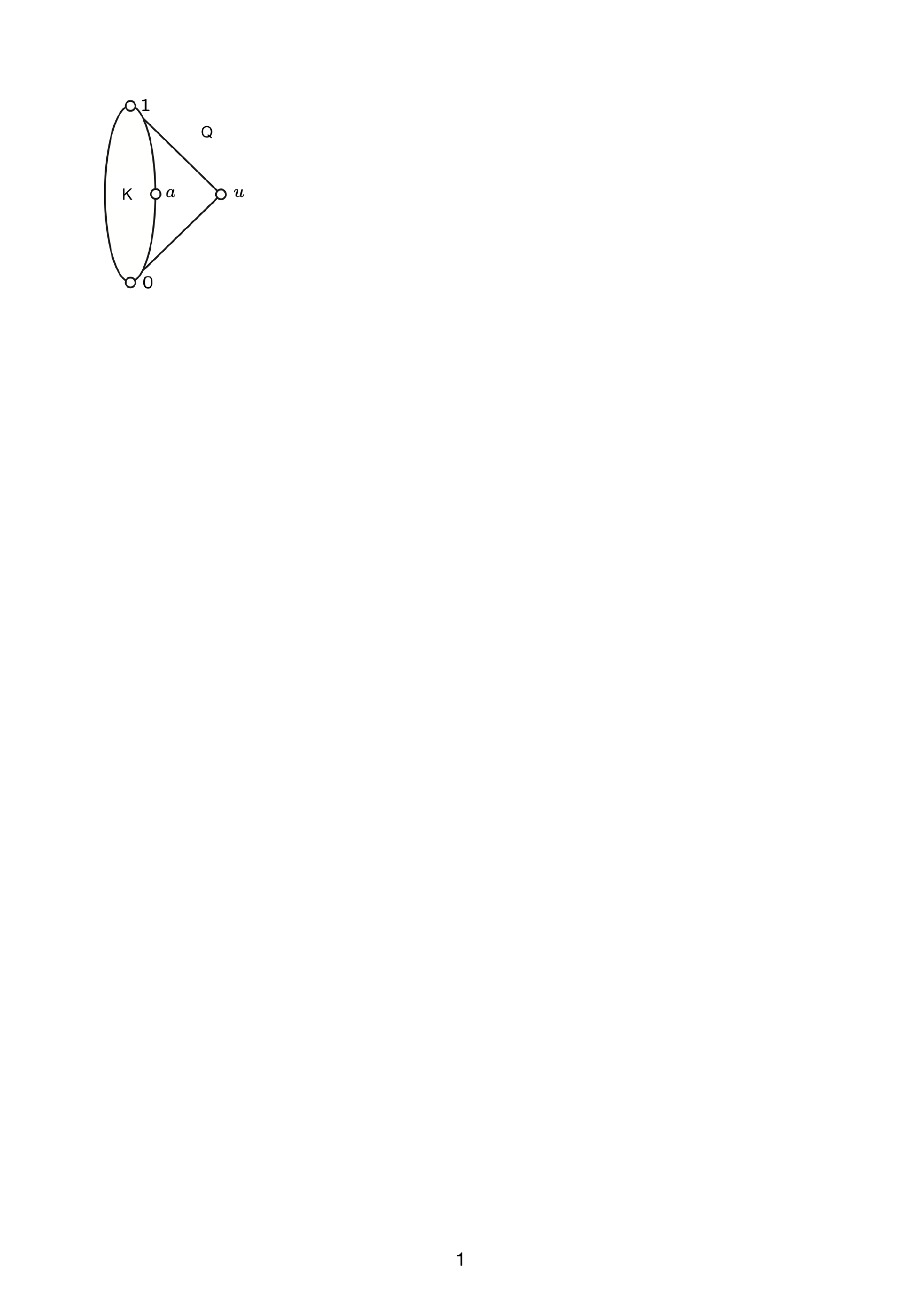} 
	      \vspace{-39\baselineskip}
	   \caption{The partial lattice $Q$}
	   \label{fig: Gratzer1}
	\end{figure}	
	
To construct and describe the lattice $F(Q)$ freely generated by $Q$, Gr\"{a}tzer and Lakser  repeatedly form joins and meets of elements of $Q$, obtaining the polynomials over $Q$, which will represent elements of $F(Q)$. For the polynomials $A$ and $B$ over $Q$, the expression $A \leq B$ denotes the relation forced by the lattice axioms and the structure of $Q$. One observes that given any polynomial $A$, there is a largest element $A_{*}$ of $K$ with $A_{*} \leq A$ and a smallest element $A^*$ of $K$ such that $A^* \ge A$. An easy computation shows that the only complement of $u$ is $a$.

The authors also show that if $K$ is almost uniquely complemented, then the only other complemented pairs in $F(Q)$ are the complemented pairs in $K$. Thus if $a$ does not have a complement in $K$, they get an almost uniquely complemented $\{0,1\}$-extension in which $a$ has a complement. By transfinite induction on the set of noncomplemented elements of $K$, they get an almost uniquely complemented $\{0,1\}$-extension $K_{1}$ of $K_{0} = K$, where each element of $K_{0}$ has a complement. Then, by a countable induction, they get a uniquely complemented $\{0, 1\}$-extension $K_{\omega}$ of $K_{0} = K$.

\subsection{Adam and Sichler theorems}

Already in the 1980s, as we have learned from Sali\u{\i}, Adam and Sichler (see \cite{Ada}) stated, among others, two significant theorems:

\begin{thm}
There is a system $\mathcal{V}$ of $2^{\aleph_{0}}$ varieties of lattices such that, for $V \in \mathcal{V}$, every $L \in V$ is a sublattice of a uniquely complemented lattice $L' \in V$.
\end{thm}

In fact, rather than proving this theorem, the authors proved the following stronger result:

\begin{thm}
There is a system $\mathcal{V}$ of $2^{\aleph_{0}}$ varieties of lattices such that, for $V \in \mathcal{V}$, if $L \in V$ is a $\{0, 1\}$-lattice, each element of which has at most one complement, then $L$ is a $\{0, 1 \}$-sublattice of a uniquely completed $\{0, 1\}$-lattice $L' \in V$.
\end{thm}

But what interests us here is how they approached the construction of a Dilworth lattice. This is done as follows:

From a lattice $L$ (see Fig. 10, left), they extracted what they call a "partial lattice" $L_{w}$ (see Fig. 10, right) by excluding the least element and the greatest element of $L$.

\begin{figure}[h] 
\vspace{-1\baselineskip}  
 \hspace{5\baselineskip}	
    	   \includegraphics[width=7in]{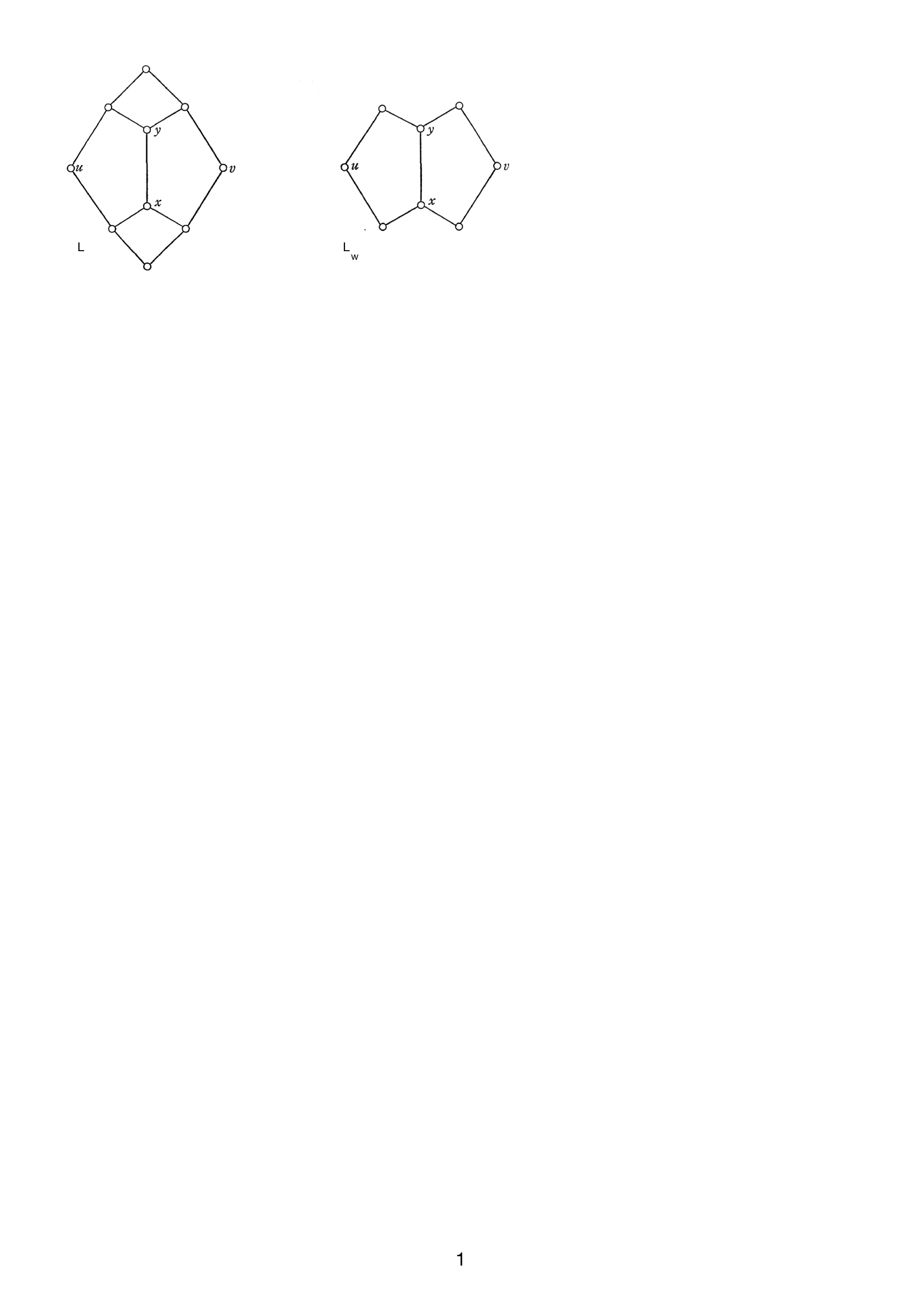} 
	      \vspace{-39\baselineskip}
	   \caption{Finite lattice $L$ and the associated partial lattice $L_{w}$}
	   \label{fig: Adam11}
	\end{figure}

A bounded lattice is a lattice with a least element 0 and a greatest element 1. A bounded lattice in which the zero and unit elements are considered as distinguished elements is called, as previously, a $\{0,1\}$-lattice.

Now, for an arbitrary $\{0,1\}$-lattice $K,\ \mathcal{L}(K)$ will denote the lattice obtained by inserting the lattice $K$ in the interval [$x,y$] of $\mathcal{L}$ (see Fig. 11).

Formally, let $\leq$ be a relation defined on $\mathcal{L} \cup (K \setminus \{0, 1\})$ by the following:

(i) For $a,b \in \mathcal{L}, a \leq b$ if and only if $a \leq b \in \mathcal{L}$;

(ii) For $a, b, \in K\setminus \{0,1\},a \leq b$ if and only if $ a \leq b \in K$;

(iii) For $a \in K\setminus \{0, 1\}, x < a < y$.

$\mathcal{L}(K)$ is the lattice whose order relation is the transitive closure of $\leq$. Let $\mathcal{L}_{w}(K)$ denote the partial lattice obtained from the lattice $\mathcal{L}(K)$ by excluding its least element and its greatest element (see Fig. 11). Observe that, in the notation of Adam and Sichler, $\mathcal{L}$ and $\mathcal{L}_{w}$ are $\mathcal{L}(2)$
and $\mathcal{L}_{w}(2)$, respectively, where 2 denotes the two-element chain.

\begin{figure}[h] 
\vspace{-1\baselineskip}  
 \hspace{2\baselineskip}	
    	   \includegraphics[width=8in]{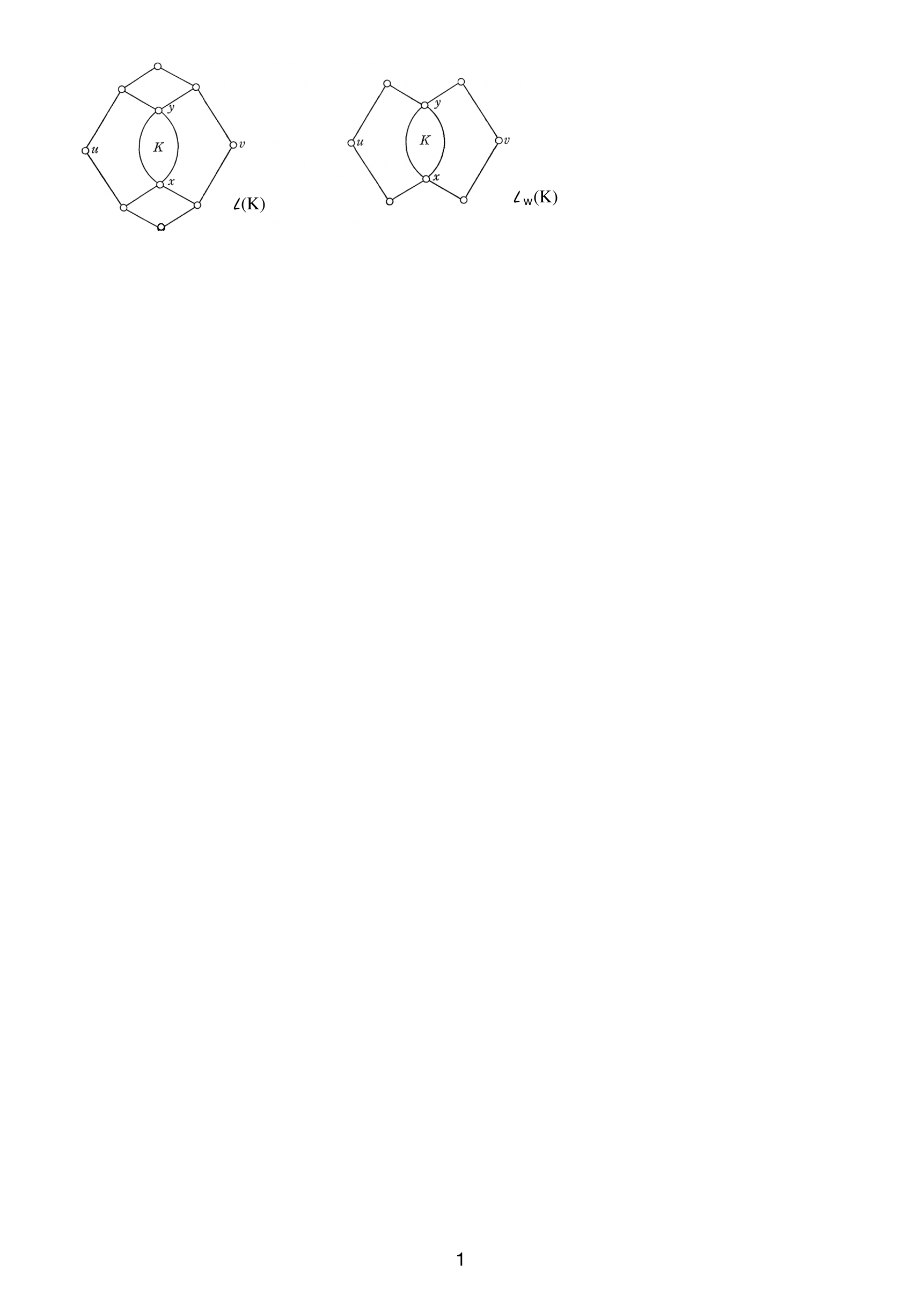} 
	      \vspace{-47\baselineskip}
	   \caption{The lattice $\mathcal{L}(K)$ and the partial lattice $\mathcal{L}_{w}(K)$ obtained from it}
	   \label{fig: Adam22}
	\end{figure}

Then, the authors proved the following lemma:

\begin{prop}
 If, for any $\{0, 1\}$-lattice $K,\ \mathcal{L}_{w}(K)$  is a relative sublattice of a lattice $L$ then the sublattice of $L$ generated by $\mathcal{L}_{w}(K)$
is isomorphic to $\mathcal{L}(K)$.
\end{prop}

Adam and Sichler's approach, though different from Gr\"{a}tzer's one, leads to a kind of Dilworth Lattice.

\section{Conclusion}

We may therefore conclude with a few observations. First, looking for a concrete realization of the Dilworth lattice, we have not found, inside mathematics, a construction of a nondistributive uniquely complemented lattice that is not freely generated. It seems that there are no natural examples of a nondistributive uniquely complemented lattice in topology, or in some other mathematical domain. As for possible physical instantiations, they are also lacking. From gluons to galaxies (see \cite{Cre}), from lattice vacuum (see \cite{Mit}) to a lattice universe (see \cite{Kir}; \cite{Duf}),  there is a lot of lattice-based structures in physics. However, the lattices studied within the framework of nonlinear lattice theory (see \cite{Dix}) do not fit the problem. In classical mechanics they reduce to one dimensional lattices - in fact, chains - of particles with nearest neighbor interactions (see \cite{Tod}, 1). In quantum physics, the use of lattices is often limited to $\mathbb{Z}_{2}$ (see \cite{Naa}) and the modeling of lattice gas in chemistry does not go further (see \cite{Sim}). Since all lattice-like biological or engineering structures (see \cite{Pum}) are finite, one cannot expect any information on the subject from the life sciences either. In conclusion, the UCN lattice of Dilworth remains an enigma. A generalization of the problem to partially ordered sets was carried out by Waphare and Joshi (see \cite{Wap}), and more recently by Chajda et al. (see \cite{Cha}), without casting any particular light on the matter, and above all, without finding an object to associate with the poset in question. We only get negative answers to the three problems posed in \cite{Gra2}.  We are in front of a true mathematical mystery.

We return now to the non-mathematical problem from which we started: if we want to define an object $O$ whose properties $P \subset E$ are not accessible to us, from the properties $E-P$ of an object $M$ that we actually know, we are not required to assume the existence of a Boolean structure on $E$. If $E$ is not finite, it suffices to assume that $E$ is a nondistributive uniquely complemented lattice.

 If the world (both known and unknown) conforms to the structure of this space $E$, then we can approximate the properties of the unknown from the properties of the known, inside this mysterious non Boolean Dilworth structure.

\end{document}